\documentclass[reqno]{amsart}
\usepackage{stmaryrd}
\usepackage{enumitem}
\usepackage{amssymb,amsmath}
\usepackage{mathabx}
\usepackage{mathbbol}
\def\dispace{\setlength{\itemsep}{2pt}}

\def\pSkip{\vskip 1mm \noindent}

\def\add{\vee}

\def\mlt{+}

\def\Id{\operatorname{Id}}

\def\Mat{\tM}

\def\MnT{\Mat_n(\Trop)}
\def\Mn{\Mat_n}

\def\STn{\Mn(\STrop)}

\def\SD{\Dl_{\operatorname{div}}}

\newcommand{\tpF}{f}

\def\ef{f^{\operatorname{es}}}
\def\tlef{\widetilde{f^{\operatorname{es}}}}

\def\eg{g^{\operatorname{es}}}

\def\bfi{ \textbf{i}}
\def\bfj{ \textbf{j}}
\def\bfk{ \textbf{k}}
\def\bfa{ \textbf{a}}
\def\bfp{ \textbf{p}}
\def\bfq{ \textbf{q}}

\newcommand{\Net}{\mathbb N}

\newcommand{\hlt}[1]{\textbf{#1}}
\newcommand\vx[1]{#1} 

\def\fuij{f^u_{i,j}}
\def\fvij{f^v_{i,j}}
\def\fwij{f^w_{i,j}}

\def\mfA{\mathfrak{A}}
\def\mfB{\mathfrak{B}}
\def\mfAgm{\mfA^\gm}
\def\mfBgm{\mfB^\gm}

\def\N{\mathbb N}

\def\nxn{n\times n}

\def\w{k}

\def\N{\mathbb N}

\def\tS{\mathcal S}
\def\Om{\Omega}
\def\om{\omega}

\def\dl{\delta}

\def\ver{\mathcal V}
\def\arc{\mathcal E}
\def\varX{\mathcal A}

\def\wlen{\ell}

\def\tlf{\widetilde f}
\def\tlg{\widetilde g}

\newcommand\corref[2]{\pSkip\textbf{Corollary #1. }\emph{#2}\vskip 1mm}

\def\tM{\mathcal M}

\newcommand\eval[2]{\left\llbracket  #1 , #2 \, \right  \rrbracket}
\newcommand\evalb[2]{\big\llbracket  #1 , #2 \, \big  \rrbracket}
\newcommand\substit[2]{\left[ #1 , #2 \, \right]}
\newcommand\sid[2]{\langle #1 , #2 \rangle}

\newcommand\cset[1]{\langle #1 \rangle}

\def\set2{{\cset{2}}}

\def\l2{x^2y^2x}
\def\ll2{yx^2y^2x}

\newcommand{\len}[1]{\operatorname{\ell}(#1)}

\newcommand{\ds}[1]{\ {#1} \ }

\def\pth{\gm}
\def\conf{\mathfrak{C}}

\def\one{\mathbb{1}}
\def\zero{\mathbb{0}}

\newtheorem{theorem}{Theorem}[section]
\newtheorem{proposition}[theorem]{Proposition}
\newtheorem{definition}[theorem]{Definition}

\newtheorem{lemma}[theorem]{Lemma}

\newtheorem{notation}[theorem]{Notation}
\newtheorem{corollary}[theorem]{Corollary}

\newtheorem{remark}[theorem]{Remark}

\newcommand {\junk}[1]{}
\newcommand{\etype}[1]{\renewcommand{\labelenumi}{(#1{enumi})}}
\def\eroman{\etype{\roman} \dispace }
\def\ealph{\etype{\alph} \dispace}

\def\({\left(}
\def\){\right)}

\newcommand{\To}{\longrightarrow }

\def\al{\alpha}
\def\bt{\beta}
\def\gm{\gamma}
\def\Gm{\Gamma}
\def\lm{\lambda}
\def\Lm{\Lambda}
\def\Dl{\Delta}
\def\Sig{\Sigma}

\def\e{\varepsilon}
\def\w{\om}

\def\Real{\mathbb R}

\def\Trop{\mathbb T}

\def\STrop{\mathbb {ST}}

\def\Auto{{\digr}}

\def\htA{\widehat A}
\def\htB{\widehat B}

\title[Semigroup identities of supertropical matrices]
 {Semigroup identities of supertropical matrices}

\author{Zur Izhakian}

\address{  Institute  of Mathematics,
 University of Aberdeen, AB24 3UE,
Aberdeen,  UK.
    }
    \email{zzur@abdn.ac.uk}

\author{Glenn Merlet}
 \address{Aix Marseille Univ, CNRS, Centrale Marseille, I2M, Marseille, France}
 \email{glenn.merlet@univ-amu.fr}

\subjclass[2010]{Primary:  20M05, 20M07, 20M30, 47D03; Secondary: 16R10,  68Q70,
14T05. }

\date{\today}

\keywords{Tropical (max-plus) matrices, supertropical matrices, idempotent
semirings, semigroup identities,  semigroup
representations, weighted diagraphs.}

\thanks{\textbf{Acknowledgment}:
  The writing of this paper was completed under the auspices of the Resarch in Pairs program of the Mathematisches Forschungsinstitut Oberwolfach, Germany. The authors thank MFO for the excellent working environment.}

\def\Id{\operatorname{Id}}

\def\digr{{\operatorname G}}
\def\hdigr{{\widecheck{\digr}}}

\begin{document}

\begin{abstract}
We  prove that, for any $n$, the monoid of all $n \times n$ supertropical matrices extending tropical matrices satisfies nontrivial semigroup identities. These identities are carried over to walks on labeled-weighted digraphs with double arcs.

\end{abstract}

\maketitle



\section*{Introduction}
\numberwithin{equation}{section}

The tropical (max-plus) semiring  is the set $\Trop\ := \Real\cup\{-\infty\}$ equipped with the operations
of maximum and summation
$$a \add  b :=\max\{a,b\},\qquad a \mlt b := \operatorname{sum}\{a,b \},
$$
serving respectively as addition and multiplication.
It extends to the supertropical semiring $\STrop$, defined over the disjoint union $\STrop := \Real \cup \{ - \infty \} \cup  \Real^\nu$ of two copies of $\Real$, together with $\zero := - \infty$, see  \cite{Iz,nualg,IR0}. The members of ~$\Real^\nu$ are denoted by $a^\nu := \nu(a)$, for $a \in \Real$, where $\nu$ is the projection  $\nu: \STrop \twoheadrightarrow \Real^\nu \cup \{ \zero \}$.
$\STrop$~ has the total ordering $\zero < a < a^\nu < b < b^\nu < \cdots  $ for any $a < b $ in $\Real$, which  determines the   addition
$$x \add y : = \left\{\begin{array}{ll}
                    \max\{x,y\}, &  \text{ if } \nu(x) \neq \nu(y), \\[1mm]
                    \nu(x), &  \text{ if } \nu(x) = \nu(y).
                  \end{array}
 \right. $$
 The multiplication of $\STrop$ is given by
 $$
  \qquad x\cdot  y : = \left\{\begin{array}{ll}
                    \operatorname{sum}\{x,y \}, &  \text{ if } x,y \in \Real,  \\[1mm]
                    \operatorname{sum}\{\nu(x), \nu(y) \}, &  \text{ otherwise},
                  \end{array}
 \right.
 $$
where the identity element is $\one := 0$. In this extension $x \add x = \nu(x)$ for every $x \in \STrop$, and thus $\STrop$ is not additively idempotent   as  $\Trop$, i.e., $x \add  x = x$. With this arithmetic
 $\Trop^\nu := \Real^\nu \cup \{ \zero\}$ is a semiring ideal, isomorphic to $\Trop$.
 Elements $x , y \in \STrop $ are $\nu$-equivalent, written  $x \cong_\nu y$, iff $\nu(x) = \nu(y)$.

All $\nxn$ matrices over any  semiring $R$, and in particular over $\Trop$ and $\STrop$,  form a multiplicative monoid~$\Mn(R)$, whose    multiplication is induced from the operations of $R$ in the familiar way. In \cite[~Theorem 3.7]{IzhakianMerletIdentity} we proved that the matrix monoid $\Mn(\Trop)$ admits nontrivial semigroup identities, completing former partial results  \cite{trID,mxID,IzMr,Shitov,Okninski,l17} which have been also dealt in  \cite{Sapir,DJK,JT}.
 In this paper  we generalize this result, and prove that the same holds for $\Mn(\STrop)$ as well.
\corref{\ref{cor:ST.IdExistence}}{The monoid $\Mn(\STrop)$ satisfies a nontrivial semigroup identity, for every $n \in \N$.}

\noindent
An explicit inductive construction of semigroup identities admitted by $\Mn(\Trop)$ has  been  provided  in ~ \cite{IzhakianMerletIdentity}, relying on identities of nonsingular matrices \cite{trID}. Theorem \ref{thm:ST.IdExistence} proves that semigroup identities for $\Mn(\STrop)$ can be composed from semigroup identities of $\Mn(\Trop)$.

These results are a further step in the study of supertropical matrices \cite{Iz,IR0,IR1,IR2,IR3,IR4}, and have  immediate consequences in representations of semigroups and of topological-combinatorial objects.
 These matrices pave the way to a new type of  linear representations of semigroups, enhancing former representations \cite{plc}, and establish a tool to extract  semigroup identities.
\corref{\ref{cor:rep.super.id}}{Any semigroup that  is faithfully represented  in  $\Mn(\STrop)$  satisfies a nontrivial semigroup  identity.}

While tropical matrices correspond uniquely to weighted digraphs (see e.g.~\cite{Butkovic,MPatW}), which play a central role in many pure and applied subjects of study, supertropical matrices correspond to weighted digraphs with possible double arcs (e.g. quivers).  By this  correspondence,  we conclude that the  semigroup identities of  $\Mn(\STrop)$  are carried over to walks on labeled-weighted digraphs with double arcs (Corollary \ref{cor:grph.walks}).
Besides digraphs, supertropical matrices enable to frame additional more complicated combinatorial objects, such as  matroids, lattices, and simplicial complexes \cite{MatI,MatII}.

\section{Preliminaries}\label{sec:2}
 We begin with a brief relevant background on semigroup identities, tropical matrices, and tropical polynomials.

\subsection{Semigroup identities}

The free monoid $\varX^*$ of finite \hlt{words} is the monoid generated by a finite \hlt{alphabet} $\varX$ -- a finite set   of \hlt{letters} -- whose identity element is the empty word. The \hlt{length} $\wlen(w)$ of a word $w$ is the number of its letters.  Excluding the empty word from $\varX^*$, the free semigroup~ $\varX^+$ is obtained.

A (nontrivial) \hlt{semigroup identity} is a formal equality $u = v$, written as a pair $\sid{u}{v}$, (with
$u \neq v$)  in ~ $\varX^+$, cf. \cite{SV}. Allowing $u$ and $v$ to be the empty word as well, i.e.,  $u,v \in \varX^*$, a monoid identity is received.
An identity  $\sid{u}{v}$ is an \hlt{$n$-letter identity}, if $u$ and~$v$ together involve  at most~$n$ different letters from~$\varX$. The  \hlt{length} of  $\sid{u}{v}$ is defined to be  $\max\{ \wlen(u), \wlen(v)\}$.

 A semigroup $\tS := (\tS, \cdot \;)$ \hlt{satisfies a
semigroup identity}~$\sid{u}{v}$,   if
\begin{equation}\label{eq:s.id}
\text{ $ \phi(u)= \phi(v)$ \ for every semigroup homomorphism $\phi
:\varX^+ \longrightarrow \tS$.}
\end{equation}
 $\Id(\tS)$ denotes the set of all semigroup identities satisfied by $\tS$.

\begin{theorem}[{\cite[{Theorem 3.10}]{trID}}]\label{thm:2id} A semigroup that satisfies an
 $n$-letter identity, $n \geq 2$,  also satisfies  a
$2$-letter identity of the same length.
\end{theorem}

Therefore, in this view, regarding the existence of semigroup identities, \hlt{we assume that $\varX = \{ a,b \}$}.

\begin{notation}[{\cite[Notation 1.2]{IzhakianMerletIdentity}}] Given a word $w \in  \varX^+$ and elements $s,t \in \tS$,  we write
  $w\eval{s}{t}$ for the evaluation of $w$ in $\tS$, obtained by substituting
   $a \mapsto s$, $b \mapsto t$, where  $s,t\in\tS$. Similarly, we write
  $\sid{u}{v}\eval{s}{t}$ for the pair of evaluations  $u\eval{s}{t}$ and $v\eval{s}{t}$  in $\tS$ of the words $u$ and $v$.

 In the certain case that $\tS = \varX^+$, to indicate that for $u,v \in \varX^+$ the evaluation  $w\eval{u}{v}$ is again a word in~ $\varX^+$, we use the particular notation $w\substit{u}{v}$. Similarly, we write
  $\sid{u}{v}\substit{s}{t}$ for $\sid{u}{v}\eval{s}{t}$.

\end{notation}

With these notations, condition \eqref{eq:s.id} can be restated as
\begin{equation*}\label{eq:s.id.2}
\text{ $\sid{u}{v} \in \Id(\tS)$ \ iff \ $u\eval{s}{t} = v\eval{s}{t}$ for every $s,t \in \tS$.}
\end{equation*}
Clearly, $\sid{u}{v} \in \Id(\tS)$ implies $\sid{u}{v}\substit{w_1}{w_2} \in \Id(\tS)$ for any $w_1,w_2 \in \varX^+$. \pSkip

For the monoid $\tS = \MnT$ of all $\nxn$ tropical matrices we have  proved:
\begin{theorem}[{\cite[Theorem 3.7]{IzhakianMerletIdentity}}]
\label{thm:IdExistence} The monoid $\MnT$ satisfies a nontrivial semigroup identity for every $n \in \N$.
\end{theorem}
\noindent In addition, \cite{IzhakianMerletIdentity} provides an  inductive construction of semigroup identities satisfied  by $\Mn(\Trop)$ which relies  on identities of nonsingular tropical  matrices \cite{trID}.


\subsection{Tropical polynomials}\label{ssec:polynimals} The semiring $\Trop[\lm_1,\dots,\lm_m]$ consists of polynomials in $m$ variables of the form
\begin{equation}\label{eq:polyToFunc1} \tpF \  = \ \bigvee_{\bfi \in
\Omega} \al_\bfi  \lm_1^{i_1}  \cdots  \lm_m^{i_m},
\end{equation}
where $\Omega\subset\N^{m}$ is a finite nonempty set of multi-induces
$\bfi =(i_1,\dots,i_m)$ for which $\al_\bfi \neq \zero$.  A polynomial ~$f$ is say to be \textbf{flat}, if  all  $\al_\bfi = \bt$ for  some  fixed  $\bt \in \Trop$.  By substitution, each polynomial $f \in \Trop[\lm_1,\dots,\lm_m]$ determines  a function $$\tlf:\Trop^{m} \longrightarrow \Trop, \qquad  (a_1, \dots, a_m) \mapsto  \tlf(a_1, \dots, a_m) :=  \bigvee_{\bfi \in
\Omega} \al_\bfi  a_1^{i_1}  \cdots  a_m^{i_m}.$$
It is well known that the map $f \mapsto \tlf$ is not injective, and that $f$ can be reduced to have only those monomials needed to describe $\tlf$.

Writing a polynomial $f = \bigvee_{\bfi} f_\bfi$ as a sum of monomials, we say that a monomial $f_\bfj$ is \textbf{inessential}, if $\tlf_\bfj(\bfa ) < \tlf(\bfa)$ for every $\bfa = (a_1, \dots, a_m) \in \Real^m$, cf. \cite{Iz}. A monomial $f_\bfj$ is \textbf{essential}, if
$\tlf_\bfj(\bfa) >  \bigvee_{\bfi \neq \bfj } \tlf_\bfi(\bfa)$ for some $\bfa \in \Real^m$,  and is \textbf{quasi-essential}, if it is inessential but
$\tlf_\bfj(\bfa) =  \bigvee_{\bfi} \tlf(\bfa)$ for some $\bfa \in \Real^m$. The \textbf{essential part} $\ef$ of~
 $f$ consists of all those monomials that are essential for $f$. The polynomial $\ef$ is unique  \cite[Lemma~ 1.6]{IzMr}, with $\tlf  = \tlef$ \cite[Proposition 1.5]{IzMr}. Therefore, for any $f,g \in \Trop[\lm_1, \dots, \lm_m]$,
 $\tlf = \tlg$ iff $\ef = \eg$.

For each monomial $f_\bfi$  of a polynomial $f$ of the form  \eqref{eq:polyToFunc1} there is the degree map
\begin{equation}\label{eq:deg}
 \dl: f_\bfi \longrightarrow \Net^m,  \qquad   \al_\bfi  \lm_1^{i_1}  \cdots  \lm_m^{i_m} \mapsto (i_1, \dots i_m).
\end{equation}
The \textbf{Newton polytope} $\Dl(f)$ of a polynomial $f   =  \bigvee_{\bfi \in
\Omega} \al_\bfi  \lm_1^{i_1}  \cdots  \lm_m^{i_m}$ is the convex hull of the set
$\{ \dl(f_\bfi) \ds| f_\bfi \text{ is a monomial of }  f \}$,
i.e., the convex  hull of the set of multi-indices $\Om \subset \N^m$. This lattice polytope has a \textbf{subdivision}
\begin{equation}\label{eq:subdiv}
\SD (f):  \Dl(f) = \Dl_1 \cup \dots \cup \Dl_\ell
\end{equation}
into disjoint lattice polytopes $\Dl_1, \dots, \Dl_\ell$, determined by projecting the upper part of the convex hull of the points  $(i_1, \dots, i_m, \al_{\bfi}) \in  \Real^{m+1}$ onto $\Dl(f) \subset \Real^m$.

The subdivision of $\Dl(f)$ yields a duality between faces of $\SD(f)$ and faces of the tropical hypersurface determined by $\tlf$ (cf. \cite{Gat}), and consequently a one-to-one correspondence between essential monomials of~ $f$ and vertices of $\SD(f) $, i.e.,  vertices of the $\Dl_i$'s in \eqref{eq:subdiv}. This gives the condition that a monomial~ $f_\bfi$ is essential for $f$ iff $\dl(f_\bfi)$ is a vertex of $\SD(f) $, and furthermore
$\tlf = \tlg $ iff $\SD(f) = \SD(g)$, for any $f,g \in \Trop[\lm_1, \dots, \lm_m]$.

   A quasi-essential monomial $f_\bfi$ of $f$ corresponds to a lattice point of $\SD(f)$ which is not a vertex, which  implies that $f$ has at least two essential monomials ~$f_\bfj$ and $f_\bfk$ for which $\tlf_\bfj(\bfa) = \tlf_\bfi(\bfa) = \tlf_\bfk(\bfa) = \tlf(\bfa)$ for some $\bfa \in \Real^m$.
   When~ $f$ is a flat polynomial, $\Dl(f) = \SD(f)$ and $\SD(f)$ has no internal vertices. Therefore, in this case, the essential monomials of ~$f$ correspond only to vertices of $\Dl(f)$.

Let $\Lm = (\Lm_{i,j})$ and $\Sig = (\Sig_{i,j})$ be two $\nxn$ matrices whose entries are  variables $\Lm_{i,j}$ and $\Sig_{i,j}$, $i,j = 1, \dots, n$. Let $\sid{u}{v} \in \Id(\Mn(\Trop))$ be an identity of $\Mn(\Trop)$, i.e., $u\eval{A}{B}  = v\eval{A}{B}$ for any $A,B \in \Mn(\Trop)$.  This means that  $u\eval{\Lm}{\Sig}$  and $v\eval{\Lm}{\Sig}$ define the same function $\Trop^{n^2} \times \Trop^{n^2}  \to \Trop^{n^2}$, which restricts to $n^2$ entry-functions $\Trop^{n^2} \times \Trop^{n^2}  \to \Trop$. Namely,  substituting  $\Lm$ and $\Sig$ into $u$ and $v$, each entry $ u\eval{\Lm}{\Sig}_{i,j}$ and $v\eval{\Lm}{\Sig}_{i,j}$ is a polynomial in~ $2 n^2$ variables for which  $\widetilde{ u\eval{\Lm}{\Sig}}_{i,j}= \widetilde{v\eval{\Lm}{\Sig}}_{i,j}$, and thus
\begin{equation}\label{eq:SD}
\SD\big({ u\eval{\Lm}{\Sig}}_{i,j}\big) = \SD\big({v\eval{\Lm}{\Sig}}_{i,j}\big).
\end{equation}
 These polynomials, which we denote by
$$ \fuij : = u\eval{\Lm}{\Sig}_{i,j}, \qquad \fvij := v\eval{\Lm}{\Sig}_{i,j},$$
are flat polynomials of the form
\begin{equation}\label{eq:polyform}
\bigvee_{\bfi \in \Om} \Lm_{1,1}^{\ell_{1,1}} \cdots \Lm_{n,n}^{\ell_{n,n}} \Sig_{1,1}^{m_{1,1}} \cdots \Sig _{n,n}^{m_{n,n}}, \qquad \bfi = (\ell_{1,1}, \dots, \ell_{n,n}, m_{1,1}, \dots, m_{n,n}) \in \Net^{2 n^2}.
\end{equation}
Therefore, \eqref{eq:SD} reads as
\begin{equation*}
\Dl\big(\fuij \big) = \Dl\big(\fvij \big).
\end{equation*}
Note that $ \fuij$ and $\fvij$ may have quasi-essential monomials.
We write $\widetilde{\fuij}(A,B)$ for the evaluation of the function $\widetilde{\fuij}: \Trop^{n^2} \times \Trop^{n^2}  \to \Trop^{n^2}$
at $A = (A_{k,l})$, $B = (B_{k,l})$.

\section{Tropical matrices vs. digraphs}\label{ssec:digraphs}

A matrix $A = (A_{i,j})$ over any semiring $R$ is associated uniquely to the \hlt{weighted digraph}
$\digr(A) := (\ver, \arc)$ over the node set
$\ver :=\{ \vx{1}, \dots, \vx{n}\}$ with   a directed arc
$\e_{i,j} := (\vx{i}, \vx{j}) \in \arc$ of \hlt{weight}~ $A_{i,j}$ from $\vx{i}$ to $\vx{j}$  for every
$A_{i,j} \ne \zero$.  $\digr(A)$ is called the digraph of the matrix $A$, and conversely $B$ is said to be the matrix of the weighted digraph $\digr'$, if $\digr' = \digr(B)$.

A \hlt{walk} $\gm$ on~$\digr(A)$ is a sequence of arcs
$\e_{i_1, j_1}, \dots, \e_{i_m, j_m} $, with $j_{k} = i_{k+1}$ for every
$k = 1,\dots, m-1$. We write $\gm := \gm_{i,j}$ to indicate that
$\gm$ is a walk from $\vx{i} = \vx{i_1}$ to $\vx{j}=\vx{j_m}$.
The \hlt{length} $\len{\gm}$ of a walk  $\gm$  is the number of
its arcs. 
The \hlt{weight} $\w(\gm)$ of
$\gm$ is the product of the weights of its arcs,  counting repeated arcs, taken with respect to the multiplicative operation of the  semiring $R$.

\begin{definition}[{\cite[Definition 1.7]{IzhakianMerletIdentity}}]\label{def:automata}
The \hlt{labeled-weighted digraph} $\Auto(A,B)$, written \hlt{lw-digraph}, of matrices $A,B\in\Mn(\Trop)$ is the digraph over the nodes~$\ver :=\{ 1,\dots, n\}$ with a directed arc $\e_{i,j}$
from~ $i$ to~ $j$ labeled~$a$ of weight~ $A_{i,j}$   for every
$A_{i,j} \ne \zero$ and a directed arc from $i$ to $j$ labeled $b$ of weight $B_{i,j}$  for every
$B_{i,j} \ne \zero$.
A walk~$\gm=\e_{i_1, j_1}, \dots, \e_{i_m, j_m}$ on~$\Auto(A,B)$ is labeled by the sequence of arcs' labels  along ~$\pth$, from~$\e_{i_1, j_1}$ to~$\e_{i_m, j_m}$,
which is a word $w$ in~$\{a,b\}^+$. A walk labeled $w$ is often denoted by $\pth_w$.
(In particular, every walk $\pth_w$ has length~$\wlen(w)$.) The weight $\w(\gm)$ of $\gm$ is the product of its arcs' weights with respect to the multiplicative operation of the  semiring $R$..
\end{definition}
\noindent The digraph $\Auto(A,B)$ may have parallel arcs, but with different labels. Note that $\Auto(A,A) = \digr(A)$.
With this definition, for tropical matrices we have the following proposition.
\begin{proposition}\label{pr:WalkInterpret}
  Given a word $w \in \{ a,b\}^+$  and matrices $A,B \in \MnT$,
  the $(i,j)$-entry of the matrix $w\eval{A}{B}$ is the maximum over the weights of all walks $\pth_{i,j} = \pth_w$ on~$\Auto(A,B)$ from $i$ to $j$ labeled by~$w$.
\end{proposition}

We next elaborate on the result of 
Theorem \ref{thm:IdExistence},
especially enhancing the interplay between matrix identities and lw-diagraphs.
Let $\pth$ be a walk on an  lw-digraph $\digr(A,B)$. We define $\conf(\pth) = (\mfAgm,\mfBgm)$ to be a pair of matrices $\mfAgm = (\mfAgm_{i,j})$ and $\mfBgm= (\mfBgm_{i,j})$ whose $(i,j)$-entry is respectively the number of occurrences of the arc $\e_{i,j}$ in $\pth$, called  \textbf{multiplicity},  labeled $a$ and $b$. The multiplicity of an arc  is ~ $0$, if it is not included in $\pth$, or it does not exist in $\digr(A,B)$.
 We call $\conf(\pth)$ the \textbf{configuration} of~ $\pth$, and write $\pth \cong_\conf \pth'$, if $\conf(\pth') =\conf(\pth)$.  Namely, if $\pth \cong_\conf \pth'$, then $\pth$ and  $\pth'$ consist exactly of the same arcs with the same multiplicities, but not necessarily with  the same ordering.
 Then, the equivalence $\pth \cong_\conf \pth'$ implies that
$\w(\pth) = \w(\pth')$ and $\len{\pth} = \len{\pth'}$. To simplify notations, we sometimes identify $\conf(\pth)$ with the point
$$ \conf(\pth) :=  (\mfAgm_{1,1},\dots, \mfAgm_{n,n}, \mfBgm_{1,1},\dots, \mfBgm_{n,n}) \in \Net^{2 n ^2}.$$

\begin{remark}\label{rmk:newton}
  Given a word $w \in \{ a,b\}^+$, let $\pth_w$ be a walk  on $\digr(A,B)$  from $i$ to $j$ labeled by $w$. Let $f:= \fwij$ be the polynomial $w\eval{\Lm}{\Sig}_{i,j}$ of the form \eqref{eq:polyform}.       The point $\conf(\pth_w) \in  \Net^{2 n^2} $ is the image $\dl(f_\bfi)$ of a monomial~ $f_\bfi$ of $f$ for which $\w(\pth_w) = f_\bfi( A,B ).$
       Thereby  a correspondence between labeled  walks $\pth_w$ on $\digr(A,B)$ from~ $i$ to $j$ and monomials $f_\bfi$ of $f$ is obtained. (Different walks may correspond to a same monomial.) We denote by $\chi(\pth_w)$ the monomial corresponding to a walk $\pth_w$, and thus have $\dl(\chi(\pth_w)) = \conf(\pth_w)$, cf. ~\eqref{eq:deg}.

       If $\gm_w$ is a walk of highest weight, then the corresponding monomial  $f_\bfi = \chi(\pth_w)$ is either essential or quasi-essential. The latter case implies that there at least two other walks of  weight $\w(\pth_w)$ from $i$ to~ $j$ labeled $w$,  since $f_\bfi$ corresponds to a lattice point of $\Dl(f)$ which is not a vertex, cf. \S\ref{ssec:polynimals}.
\end{remark}

Informally we have a correspondence $\chi$ between walks and monomials,
given by associating an arc $\e_{i,j}$ labeled $a$ to the  variable $\Lm_{i,j}$ and an arc $\e_{i,j}$ labeled $b$ to the  variable $\Sig_{k,\ell}$.  The powers of these variables in a monomial are   the arcs' multiplicities, encoded by the configuration of a walk.
\pSkip

Combining the above graph view and the polynomial view from  \S\ref{ssec:polynimals}, we have the following result.
\begin{lemma}\label{lem:paths}
  Let $\sid{u}{v}\in\Id(\MnT)$ be an identity for $\MnT$, and let $ \digr(A,B)$ be the lw-digraph of $A,B \in \Mn$. Fix $i$ and $j$, and  let $\pth_u$ be a  walk labeled $u$  of highest weight from $i$ to $j$ on $\digr(A,B)$. Let $f := \fuij = u\eval{\Lm}{\Sig}_{i,j}$ and $g:= \fvij = v\eval{\Lm}{\Sig}_{i,j}$.
   \begin{enumerate} \eroman
     \item If $\dl(\chi(\pth_u))$ is a vertex of $\Dl(f)$, then there exists a walk $\pth_v \cong_\conf \pth_u$  from $i$ to $j$  labeled  $v$ of  weight $\w(\pth_u)$ on $\digr(A,B)$.
         \item If $\dl(\chi(\pth_u))$ is not a vertex of $\Dl(f)$, then there are at least two walks $\pth_v',\pth_v'' \not \cong_\conf \pth_u$   from $i$ to $j$  labeled  $v$ of  weight $\w(\pth_u)$.
     \item If there exists another walk $\pth_u' \not \cong_\conf \pth_u$ labeled $u$ with $\w(\pth_u) = \w(\pth'_u)$, then there are walks $\gm_v' \not \cong_\conf \pth_v$ labeled  $v$ such that $\w(\pth_v) = \w(\pth'_v) = \w(\pth_u)$.

   \end{enumerate}

 \end{lemma}

 \begin{proof} Clearly, since $(u\eval{A}{B})_{i,j} = (v\eval{A}{B})_{i,j}$, there exits a walk $\pth_v$ labeled $v$ for which   $\w(\pth_u) = \w(\pth_v)$.
 Since  $\sid{u}{v}\in\Id(\MnT)$, $f = u\eval{\Lm}{\Sig}_{i,j}$, and $g = v\eval{\Lm}{\Sig}_{i,j}$,  we have $\tlf= \tlg$ for the functions $\tlf, \tlg$;  in particular, $\tlf(A,B) = \tlg(A,B)$, and thus  $\Dl(f) = \Dl(g)$, cf. \S\ref{ssec:polynimals}. Since $\pth_u$ is of highest weight, and therefore ~$\pth_v$ is also of highest weight, there exist quasi-essential monomials $f_\bfi$, $g_\bfi$ of $f,g$ such that
  $$\tlf_\bfi(A,B) = (\chi(\pth_u))(A,B) = \w(\pth_u) = \w(\pth_v) = (\chi(\pth_v))(A,B)  = \tlg_\bfj(A,B).$$
   Let $\bfp = \dl(f_\bfi)$ and  $\bfq = \dl(g_\bfj)$ be the lattice points corresponding to these monomials in  $\Dl(f)$ and $\Dl(g)$. That is
   $\bfp = \conf(\pth_u) = \dl(\chi(\pth_u))$ and $\bfq = \conf(\pth_v) = \dl(\chi(\pth_v))$. Recall that $\Dl(f) = \Dl(g)$.
   \pSkip
   (i): If $\bfp$ is a vertex of $\Dl(f)$, then $\bfp$ is also a vertex of $\Dl(g)$. So, there is a walk $\pth_v$ for which $\bfp = \bfq$, and thus $\conf(\pth_u) = \conf(\pth_v)$. Hence $\pth_v \cong_\conf \pth_u$.

   \pSkip
   (ii):  If $\bfp$ is not a vertex of $\Dl(f)$, then $f_\bfi$ is a quasi-essential monomial of $f$, cf. \S\ref{ssec:polynimals}, and $f$ has at least two essential monomials $f_\bfj$ and $f_\bfk$ corresponding to vertices of $\Dl(f)$  for which $\tlf(A,B) =  \tlf_\bfj(A,B) = \tlf_\bfi (A,B) =\tlf_\bfk(A,B)$. Each of these monomials is associated to a walk labeled $u$ of weight $\w(\pth_u)$. Thus, by part (i), there are two walks $\pth'_v, \pth''_v$ labeled $v$ with $\w(\pth'_v) = \w(\pth''_v) = \w(\pth_u)$.

   \pSkip
   (iii): If there are two walks  $\pth_u \not \cong_\conf \pth_u'$ labeled $u$ of highest weight, then each walk is associated to a different monomial of~$f$. These monomials correspond to different lattice points $\bfp$ and $\bfp'$  in $\Dl(f)$.
   If both $\bfp$ and $\bfp'$ are vertices of $\Dl(f)$,  then we are done by part (i). Otherwise,  one of these lattice points   is not a vertex of $\Dl(f)$,  and the proof  follows from part (ii).
 \end{proof}

\section{Semigroup identities of supertropical matrices}\label{sec:5}

Matrices over $\STrop$ are matrices with entries in $\STrop$; the monoid of these matrices is denoted by $\Mn(\STrop)$.
The map $\nu: \STrop \to \Trop^\nu $ extends entry-wise to  matrices, i.e., $\nu(A) = (\nu(A_{i,j}))$, inducing the $\nu$-equivalence~$\cong_\nu$ on matrices   in $\STn$. By the structure of $\STrop$, for any matrix $A \in \STn$ there exists a matrix $\htA$ with entries in $\Trop = \Real \cup \{ \zero \}$ such that $\htA \cong_\nu A$.

Similar to  matrices in $\MnT$, cf. \S\ref{sec:2}, each supertropical matrix $A \in \Mn(\STrop)$ is associated with a labeled-weighted diagraph $\digr(A)$, while now arcs  have their weights $\STrop$.  A walk $\pth$ on $\digr(A)$ has weight $\w(\pth)$ in $\Real^\nu$, if one of its arcs has a weight in $\Real^\nu$. Otherwise, $\pth$ has weight in $\Real$.  A walk $\pth$ from $i$ to $j$ is said to have highest $\nu$-weight, if $\nu(\w(\gm))$ is maximal among the weights of all other walks from $i$ to $j$. Each arc $\e_{i,j}$ of $\digr(A)$ having weight in $\Real^\nu$ can be interpreted as a double arc, i.e., two parallel arcs from $i$ to~ $j$ with the same label and the same weight  in $\Real$. $\hdigr(A)$ denotes the digraph taken with this interpretation.

\begin{lemma}\label{lem:nuEqMat.1}
 Given an identity $\sid{u}{v} \in\Id(\Mn(\Trop))$, then $u\eval{A}{B} \cong_\nu  v\eval{A}{B}$ for any   $A,B \in \STn$.
\end{lemma}
\begin{proof}  The monoid $\Mn(\Trop^\nu)$ of all matrices over $\Trop^\nu$ is isomorphic to  $\MnT$, and thus $\nu(u\eval{A}{B}) = \nu(v\eval{A}{B})$, which by definition implies $u\eval{A}{B}
\cong_\nu  v\eval{A}{B}$.
\end{proof}

\begin{remark}\label{rem:paths}
  Lemma \ref{lem:paths} remains valid for matrices in $\STn$. Indeed,  mark each arc wether its weight is in $\Real$ or in $\Real^\nu$, apply the lemma for the images  $\nu(A)$ and $\nu(B)$ of the matrices $A$ and $B$ to find the walks, then recover the walks' weight (which could be either $\Real$ or in $\Real^\nu$) from the marking of their arcs.

\end{remark}

\begin{lemma}\label{lem:nuEqMat.2}
 Given an identity $\sid{u}{v} \in\Id(\MnT)$, then  $u\eval{A}{B} = v\eval{A}{B}$ for any   $A\cong_\nu B$ in~ $\STn$.
\end{lemma}
\begin{proof}

 Let  $U := u\eval{A}{B}$ and   $V :=   v \eval{A}{B}$.  Since $A \cong_\nu B$, then
 $U\cong_\nu V$. Moreover,  $\htA  = \htB$,
  and thus
\begin{equation}\label{eq:str}
 u\evalb{\htA}{\htB}  = u\evalb{\htA}{\htA} = Z =  v\evalb{\htA}{\htA} = v\evalb{\htA}{\htB} . \tag{$*$}
\end{equation}
 Hence, $Z$ is a matrix in $\STn$, obtained as the  power $\len{u}$ of the matrix $\htA$ in $\Mn(\Trop)$. Accordingly, an entry $Z_{i,j}$ of $Z$ corresponds to a walk of highest weight on $\digr(\htA)$. Let $U' :=  u\evalb{\htA}{\htB}$ and $V' :=  v\evalb{\htA}{\htB}$, which are equal by \eqref{eq:str}.

Fixing ~$(i,j)$, the entry $U_{i,j}$ (resp. $V_{i,j}$) corresponds to walks $\Gm_{u}$ (resp. $\Gm_{v}$) from $i$ to $j$  labeled $u$ (reps.~ $v$)  of highest $\nu$-weight on $\digr(A,B)$. In particular, $\w(\pth_u) \cong_\nu \w(\pth'_u)$ for all $\pth_u , \pth'_u\in \Gm_u$. The same holds for $\Gm_v$.
 Note that,  $U'_{i,j } \in \Trop^\nu$ (resp. $V'_{i,j } \in \Trop^\nu$) implies $U_{i,j } \in \Trop^\nu$ (resp. $V_{i,j } \in \Trop^\nu$).

We need to prove  that, if $U_{i,j} \in \Trop^\nu$, then also $V_{i,j} \in \Trop^\nu$.
If $U_{i,j} = \zero$, then $U'_{i,j} = \zero$. Thus $V'_{i,j} = \zero$ by~ \eqref{eq:str}, and hence $V_{i,j} = \zero$.
If both $U_{i,j}$ and $V_{i,j}$ are in $\Real$, then $U_{i,j} = U'_{i,j} = V'_{i,j} = V_{i,j}$ by~\eqref{eq:str}.
 Otherwise, 
say $U_{i,j}$ has a value in $\Real^\nu$, for which we have two cases:
\begin{enumerate}\ealph
  \item $\Gm_{u}$ includes a walk $\pth_u$ having weight in $\Real^\nu$,
  \item $\Gm_{u}$ includes only walks $\pth_u$ having weight in $\Real$.
\end{enumerate}
Let $\fuij = u\eval{\Lm}{\Sig}_{i,j}$ and $\fvij = v\eval{\Lm}{\Sig}_{i,j}$.

\pSkip
\underline{Case (a)}: $\w(\pth_u) \in \Real^\nu$ implies that one for the arcs composing $\pth_u$, say $\e_{s,t}$, has weight in $\Real^\nu$.
If $\conf(\pth_u)$ is a vertex of $\Dl(\fuij)$, then by Lemma~\ref{lem:paths}.(i), in the view of Remark \ref{rem:paths}, $\Gm_v$ contains a walk $\pth_v$  with
$\conf(\pth_v) = \conf(\pth_u)$. This means that $\pth_v$ consists of the same arcs as ~$\pth_u$, in particular  $\e_{s,t}$ is an arc of~ $\pth_v$ with $\w(\e_{s,t}) \in \Real^\nu$, implying that $\w(\pth_v)\in \Real^\nu$. Thus, $V_{i,j} \in \Real^\nu$, and hence $V_{i,j} = U_{i,j}$.

If $\conf(\pth_u)$ is not a vertex of $\Dl(\fuij)$, then by  Lemma~\ref{lem:paths}.(ii) and  Remark \ref{rem:paths}, $\Gm_v$  contains at least two walks $\pth'_v, \pth''_v$  with $\w(\pth'_v) = \w(\pth''_v) = \w(\pth_u)$. Thus $V_{i,j} = \w(\pth'_v) + \w(\pth''_v) = \w(\pth_v)^\nu \in \Real^\nu$, and hence $V_{i,j} = U_{i,j}$.

\pSkip
\underline{Case (b)}: Since $U_{i,j} \in \Real^\nu$,  $\Gm_u$ contains at least two walks $\pth'_u, \pth''_u$ with weight in $\Real$, and thus $\pth'_u$ and $ \pth''_u$ are also walks on $\digr(\htA)$ -- the digraph of $v\evalb{\htA}{\htB}$, cf.~\eqref{eq:str}.
Therefore, $\w(\pth'_u) + \w(\pth''_u) = \w(\pth'_u)^\nu = V'_{i,j} \in \Real^\nu$, implying that  $V_{i,j} \in \Real^\nu$, and hence $U_{i,j} = V_{i,j}$.
\end{proof}

\begin{theorem}\label{thm:ST.IdExistence}
 Let $\sid{u}{v}$ and $ \sid{u'}{v'}$ be identities  for $\MnT$.
  The monoid $\STn$ admits the  semigroup identity $ \sid{u\substit{u'}{v'}}{v\substit{u'}{v'}} $.
\end{theorem}
\begin{proof}
$U' := u'\eval{A}{B} \cong_\nu  v'\eval{A}{B} := V'$  by Lemma \ref{lem:nuEqMat.1}, then $u\eval{U'}{V'} = v\eval{U'}{V'}$ by Lemma ~\ref{lem:nuEqMat.2}.
\end{proof}

\begin{corollary}\label{cor:ST.IdExistence}
  The monoid $\STn$ admits nontrivial  semigroup identities.

\end{corollary}
\begin{proof} Immediate by Theorem \ref{thm:ST.IdExistence}, since
  $\MnT$ admits nontrivial identities by Theorem
  \ref{thm:IdExistence}.
\end{proof}

Given a semiring $R$, a   (linear) representation of a semigroup $\tS$ is a semigroup homomorphism $$\rho: \tS \To \Mn(R),$$ i.e., $\rho(st) = \rho(s) \rho(t)$  for any  $s,t \in \tS.$
$\rho$ is said to be faithful, if it is injective.

\begin{corollary}\label{cor:rep.super.id}
Any semigroup which is faithfully represented  in $\STn$  satisfies a nontrivial semigroup identity.\end{corollary}

Considering the weighted digraphs $\hdigr(A)$  with double arcs  associated to matrices $A \in \Mn(\STrop)$,  Theorem \ref{thm:ST.IdExistence} receives the following  meaning.

\begin{corollary}\label{cor:grph.walks}
For any  labeled-weighted digraph $\digr$  with  arcs labeled by $\{ a,b \}$, possibly with double arcs,
there are two different words $u, v \in \{ a,b \}^+$,
such that, for any pair~$i,j$ of nodes of $\digr$,
the highest weight of walks from~$i$ to~$j$ is the same for walks labeled~$u$
and for walks labeled~$v$. In addition, if the there are a unique walk of  highest weight labeled $u$, then there is a unique walk labeled $v$.
Moreover, there is a pair of words that works for all digraphs having a given number of nodes.
\end{corollary}


\end{document}